\documentclass[english]{amsart}
\usepackage{color}
\usepackage{amssymb}

\makeatletter
 \theoremstyle{plain}
\newtheorem{thm}{Theorem}
  \theoremstyle{definition}
  \newtheorem{defn}[thm]{Definition}
  \theoremstyle{remark}
  \newtheorem{rem}[thm]{Remark}
  \theoremstyle{plain}
  \newtheorem{lem}[thm]{Lemma}
  \theoremstyle{plain}
  \newtheorem{prop}[thm]{Proposition}
  \theoremstyle{plain}
  \newtheorem{cor}[thm]{Corollary}
 \newtheorem{question}[thm]{Question}

\begin{document}

\title{Naming an indiscernible sequence in $NIP$ theories}

\author{Artem Chernikov and Pierre Simon}

\begin{abstract}
In this short note we show that if we add predicate for a dense complete indiscernible sequence in a dependent theory then the result is still dependent. This answers a question of Baldwin and Benedikt and implies that every unstable dependent theory has a dependent expansion interpreting linear order.
\end{abstract}
\maketitle

\section*{Introduction}

Let $T$ be an $NIP$ theory in a language $L$. Consider a model $M$
and a small indiscernible sequence $I$ indexed by a dense complete
linear order (small means that $M$ is $|I|^{+}$-saturated). We consider
the language $L_{P}$ with a unary predicate $P$ added for the sequence
$I$, and let $T_{P}=Th(M,I)$.

\begin{defn}
We say that an $L_{P}$-formula is \textit{bounded} if it is of the
form $$(Q_{1}x_{1}\in P)(Q_{2}x_{2}\in P)...(Q_{n}x_{n}\in P)\phi(x_{1},...,x_{n},\bar{y}),$$
where $\phi$ is an $L$-formula and each $Q_{i}$ is either $\exists$
or $\forall$.
\end{defn}
In   \cite{BB} Baldwin and Benedikt prove the following.

\begin{thm}
\label{thm: BB}Assume $T$ is $NIP$.

1) For each dense complete indiscernible sequence $I$ and formula
$\phi(\bar{a},y)$, there is some $\bar{c}\in I$ such that for every $\bar{b}\in I$,
the truth value of $\phi(\bar{a},\bar{b})$ is totally determined
by the quantifier-free order type of $\bar{b}$ over $\bar{c}$ .

2) Every formula in $T_{P}$ is equivalent to a bounded one.

From this follows :

3) $Th(M,I)\equiv Th(M,J)$ if and only if $EM(I)=EM(J)$.

4) $P$ is stably embedded and the $L_{P}$-induced structure (traces on $P$ of all $L_P$-definable sets with parameters from $M$) is that of a pure linear order.
\end{thm}
\begin{rem}
Point 1) is the Theorem 5.2 there, but for a simplified proof see \cite{Adl}, Section 3.
2) is  Theorem 3.3, 3) is Theorem 8.1, 4) is Corollary 3.6.
\end{rem}

They prove that if $T$ is stable then $T_{P}$ is stable as well,
and ask whether $T_{P}$ is always dependent when $T$ was. In the next section
we answer this question positively.
Throughout the paper we assume Martin's Axiom (MA).

\section*{Dependence of $T_{P}$}

First a trivial combinatorial observation.

\begin{lem}
\label{lem: monotonic functions}Let $h_{k}:\omega\to I$, $k \leq m$ be monotone functions, $h_{0}(n) < ... < h_{m}(n)$ for all $n$ and let $a_{1},...,a_{n}\in I$. Then for some $p\leq 2nm+1$
both $(h_{k}(p))_{k \leq m}$ and $(h_{k}(p+1))_{k \leq m}$ have the same order
type over $a_{1},...,a_{n}$.
\end{lem}
\begin{proof}
Suppose not. Then for each $p\leq 2nm+1$ there is some
$i\leq n,j\leq m$ with $h_{i}(p)<a_{j},h_{i}(p+1)\geq a_{j}$ or $h_{i}(p)>a_{j},h_{i}(p+1)\leq a_{j}$ or $h_i(p) = a_j, h_i(p+1) \neq a_j$,
 and by monotonicity for every pair of $i,j$ there can be only up to two
such $p$ - a contradiction.
\end{proof}

Next a crucial technical lemma.

\begin{prop}
\label{pro: cuts on the plane} 1) Assume $P$ is ordered by some
$L_{p}$-definable ``$<$''. Let $I=(b_{i})_{i<\omega}$ be an
$L_{P}$-indiscernible sequence and $E\subset\omega$ the set of even
numbers. Assume that $f:E\to P(\mathbb{M)}$, $n<\omega$ even and
$\phi(x_{1},...,x_{n};y_{1},...,y_{n})\in L_{p}$ such that for any
sequence $k_{1}<k_{2}<...<k_{n}\in E$ we have

$\models\phi(b_{k_{1}},...,b_{k_{n}};f(k_{1}),...,f(k_{n}))$.

Then there is $g:\omega\to P(\mathbb{M})$ extending $f$ and $l_{1},...,l_{n}\in\omega$
with $l_{i}\equiv i~(mod2)$ and $\models\phi(b_{l_{1}},...,b_{l_{n}};g(l_{1}),...,g(l_{n}))$.

2) Same claim but assuming that the $L_{P}$-induced structure on
$P$ is just the equality. 
\end{prop}
\begin{proof}
1) Since by Theorem \ref{thm: BB} the $L_{P}$-induced structure on $P$ is just that of linear order by compactness there is some $m<\omega$ such that given any ${(a}_{1},...,a_{n})\in\mathbb{M}$ there are some ${(c}_{1}<...<c_{m})\in P(\mathbb{M})$ such that for any $(d_{1},...,d_{n})\in P(\mathbb{M)}$ the truth value of $\phi(a_{1},...,a_{n};d_{1},...,d_{n})$ is totally determined by the order type of $\bar{d}$ over $\bar{c}$.

Now for each $k \leq m$ let $h_{k}:\mathbb{M}^n\to P$ be the $L_{P}$-definable
function sending $(a_{1},...,a_{n})$ to the corresponding $c_{k}$ (W.l.o.g. we assume there is a constant $\rho$ in $P$. If for some $\bar{a}$ there are $k'<k$ alternations we let $h_j(\bar a)=\rho$ for $j > k'$).

We have : 
\vspace{3pt}

({*}) for any $(a_{1},...,a_{n})\in\mathbb{M},(d_{1},...,d_{n})$,
$(d_{1}',...,d_{n}')\in P$,

$\bar{d}$ and $\bar{d}'$ have the same order type over $(h_{k}(\bar{a}))_{\leq m}$
$\implies$ $\phi(\bar{a};\bar{d})\equiv\phi(\bar{a};\bar{d}')$
\vspace{3pt}

From this by $L_{P}$-indiscernibility of $I$ :

\vspace{3pt}
({*}{*}) for any $b_{1},...{,b}_{n}$ and $b_{1}',...,b_{n}'$ increasing
sequences from $I$ and $\bar{d},\bar{d}'\in P$,

$\bar{d}$ has the same order type over $(h_{k}(\bar{b}))_{\leq m}$ as
$\bar{d}'$ over $(h_{k}(\bar{b}'))_{\leq m}$ $\implies$ $\phi(\bar{b};\bar{d})\equiv\phi(\bar{b}';\bar{d}'$).

\vspace{3pt}
Choose some $(0<l_{2}<...<l_{n-2}<l_{n}) \in E^{n/2}$ with $l_{2(i+1)}-l_{2i}> 2mn+1$.
Define $h'_{k}:\omega\to P(\mathbb{M})$ by $h'_{k}(p)=h_{k}(b_{p},b_{l_{2}},b_{l_{2}+p},...{,b}_{l_{n-2}+p},b_{l_{n}})$.

By $L_{P}$-indiscernibility of $I$ the ${h'}_{k}$'s are monotonic (at least in the interval $[1, 2mn+1]$ which is all that matters).
Thus by Lemma \ref{lem: monotonic functions} we find some (w.l.o.g. odd)
$p_{0} \leq 2mn+1$ such that ${(h'}_{k}(p_{0}))_{\leq m}$ has the same
order type as $(h'_{k}(p_{0}+1))_{\leq m}$ over $f(l_{2}),...,f(l_{n})$.
And again by $L_{P}$-indiscernibility and density of $P$ we can
find some $g(p_{0}),g(l_{2}+p_{0}),...,g(l_{n-2}+p_{0})\in P(\mathbb{M})$
such that $g(p_{0}),f(l_{2}),...,g(l_{n-2}+p_{0}),f(l_{n})$ has the
same order-type over ${(h'}_{k}(p_{0}))_{\leq m}$ as $f(p_{0}+1),f(l_{2}),...,f(l_{n-2}+p_{0}+1),f(l_{n})$
over $(h'_{k}(p_{0}+1))_{\leq m}$, and so  by ({*}{*})

$\models\phi(b_{p_{0}},b_{l_{2}},...,b_{l_{n-2}+p_{0}},b_{l_{n}};g(p_{0}),f(l_{2}),...,g(l_{n-2}+p_{0}),f(l_{n}))$
and we are done.

2) Analogously.
\end{proof}

This gives us a Ramsey-like result on completing indiscernible sequences of triangles

\begin{cor}
\label{triangles}
Let $(a_i)_{i \in \omega}\in \mathbb{M}$, $(b_{2i})_{i\in \omega}\in P$ be given and $d \in \mathbb{M}$. Then there is some sequence $(a'_ib'_i)_{i \in \omega}$ $L_P$-indiscernible over $d$ and such that for every $\psi \in L_P$ \
\

$(+)$  $\psi((a'_{2i}a'_{2i+1}b'_{2i})_{i<n},d)$ $\implies$ $\psi((a_{2k_i}a_{2k_i+1}b_{2k_i})_{i<n},d)$ for some $k_0 < k_1 <... <k_{n-1} \in \omega$.
\end{cor}
\begin{proof}
First by Ramsey find an $L_P$-indiscernible sequence $(a''_{2i}a''_{2i+1}b''_{2i})_{i<\omega}$ with property $(+)$. Now let $I=(a''_i)_{i<\omega}$ and $f(a''_{2i})=b''_{2i}$ and use Proposition \ref{pro: cuts on the plane} with compactness to conclude.
\end{proof}

Finally we are ready to prove our main result.

\begin{thm}
\label{thm:T_{P}-is-dependent}$T_{P}$ is dependent.
\end{thm}
\begin{proof}
First note that by Theorem \ref{thm: BB}, the $L_{P}$-induced
structure on $P$ is equality or it is ordered by some $L$-formula (with parameters).

We prove by induction on the number of bounded quantifiers that all
$L_{P}$-formulas are dependent, and since the set of formulas with
$NIP$ is closed under boolean combinations it is enough to consider
adding single existential bounded quantifier to a dependent formula. 

So assume $\phi(x;y)=(\exists z\in P)\psi(x,y,z)$ has $IP$ where $\psi(x,y,z)$
is an $L_{P}$-formula. Then there is some $L_{p}$-indiscernible
sequence $(a_{i})_{i<\omega}$ and $d$ such that $\phi(d,a_{i})$
holds if and only if $i$ is even, and so for $i=2k$ let $b_{i}\in P$
be such that $\psi(d,a_{i},b_{i})$ holds. By Lemma \ref{triangles} we find some
sequence $(a'_ib'_i)_{i<\omega}$ which is $L_{P}$-indiscernible and (using $(+)$) still $\psi(d,a'_{2i},b'_{2i})$  and $\neg \psi(d,a'_{2i+1},b'_{2i+1})$ hold. But this means that $\psi(d;y,z)$ has infinite alternation - contradicting the inductive assumption. 
\end{proof}

\begin{question}
Assuming $T$ is strongly-dependent, is $T_P$ strongly-dependent ?
\end{question}

\begin{rem}
Note however that unsurprisingly $dp$-minimality is not preserved in general after naming an indiscernible sequence. By \cite{Good}, Lemma 3.3, in an ordered $dp$-minimal group, there is no infinite definable nowhere-dense subset, but of course every small indiscernible sequence is like this.
\end{rem}

\begin{cor}
Every unstable dependent theory has a dependent expansion interpreting
an infinite linear order.
\end{cor}
\begin{proof}
Just take a small indiscernible sequence that is not an indiscernible
set, mark it by a predicate and use Theorem \ref{thm:T_{P}-is-dependent}.
\end{proof}

\section*{Acknowledgements}

We are greatful to Ehud Hrushovski for telling us to ``look at
cuts on the plane'', to Itay Kaplan for numerous attempts to discourage
our belief in the facts proved here and to the organizers of the Banff
meeting on ``Stable methods in unstable theories'' during which
this work was essentially accomplished.

\bibliographystyle{alpha}
\bibliography{BB3}

\end{document}